\documentclass[11pt]{amsart}

\usepackage{amsmath, amsthm,afterpage}

\usepackage{amssymb,amscd,amstext,units,mathtools,color,xcolor,enumerate,bm,pifont,comment}
\usepackage[mathscr]{eucal}
\usepackage[T1]{fontenc}
\usepackage[utf8]{inputenc}
\usepackage[all]{xy}
\usepackage[hidelinks]{hyperref}

\usepackage{geometry}

\theoremstyle{plain}
\newtheorem*{cor*}{Corollary}
\newtheorem*{prop*}{Proposition}
\newtheorem*{thm*}{Theorem}
\newtheorem*{notation*}{Notation}
\newtheorem{thm}{Theorem}[section]
\newtheorem{cor}[thm]{Corollary}

\newtheorem{fact}[thm]{Fact}

\theoremstyle{definition}

\newtheorem{question}[thm]{Question}

\numberwithin{equation}{subsection}

\newcommand{\forkindep}[1][]{%
  \mathrel{
    \mathop{
      \vcenter{
        \hbox{\oalign{\noalign{\kern-.3ex}\hfil$\vert$\hfil\cr
              \noalign{\kern-.7ex}
              $\smile$\cr\noalign{\kern-.3ex}}}
      }
    }\displaylimits_{#1}
  }
}  

\newtheorem{innercustomgeneric}{\customgenericname}
\providecommand{\customgenericname}{}
\newcommand{\newcustomtheorem}[2]{%
  \newenvironment{#1}[1]
  {%
   \renewcommand\customgenericname{#2}%
   \renewcommand\theinnercustomgeneric{##1}%
   \innercustomgeneric
  }
  {\endinnercustomgeneric}
}

\newcustomtheorem{customthm}{Theorem}

\title{Extensions of definable local homomorphisms in o-minimal structures and semialgebraic groups}

\author{Eliana Barriga}
\address{ Eliana Barriga\\ Department of Mathematics, Ben Gurion University of the Negev, Be’er-Sheva 84105, Israel}
\email{barrigat@post.bgu.ac.il}


\thanks{Supported by the Israel Science Foundation (ISF) grants No. 181/16 and 1382/15.}

\begin{document}

\begin{abstract}
We state conditions for which a definable local homomorphism between
two locally definable groups $\mathcal{G}$, $\mathcal{G^{\prime}}$
can be uniquely extended when $\mathcal{G}$ is simply
connected (Theorem \ref{T:ExtofdefLocHomoSCGp}). As an application
of this result we obtain an easy proof of \cite[Thm. 9.1]{BADefCompIJM}
(see Corollary \ref{C:ExtlocHomLDGpsTFcase}). We also prove that
Theorem 10.2 in \cite{BADefCompIJM} also holds for any definably
connected definably compact semialgebraic group $G$ not necessarily
abelian over a sufficiently saturated real closed field $R$; namely,
that the o-minimal universal covering group $\widetilde{G}$ of $G$
is an open locally definable subgroup of $\widetilde{H\left(R\right)^{0}}$
for some $R$-algebraic group $H$ (Thm. \ref{T:UniveCoverCompactGeneral}).
Finally, for an abelian definably connected semialgebraic group $G$
over $R$, we describe $\widetilde{G}$ as a locally definable extension
of subgroups of the o-minimal universal covering groups of commutative
$R$-algebraic groups (Theorem \ref{T:UniveCoverAbelianGeneral}).
\end{abstract}

\maketitle

\section{Introduction and Preliminaries}\label{Introduction}
The study of definable and locally definable groups has been of importance
in the research in model theory of o-minimal structures, and includes such classes as the semialgebraic and the subanalytic groups. The ordered
real field $\left(\mathbb{R},<,+,\cdot\right)$ as well as its expansion
with the exponential function are examples of o-minimal structures
\cite{LVD}.

\textit{Let $\mathcal{M}$ be a sufficiently $\kappa$-saturated o-minimal
structure. By definable we mean definable in $\mathcal{M}$. }

A group is called \textit{locally definable} if the domain of the
group and the graph of the group operation are a countable unions of
definable sets. Every $n$-dimensional locally definable group $\mathcal{G}$ can be endowed with a unique \textit{topology} $\tau$ making the
group into a topological group such that any $g\in \mathcal{G}$ has a definable neighborhood definably isomorphic to $M^{n}$
\cite[Prop. 2.2]{PS00}. From now on, any topological property on
a locally definable group refers to this $\tau$-topology, unless
stated otherwise.

When a locally definable group is definable, its $\tau$-topology
agrees with the t-topology given by Pillay in \cite[Prop. 2.5]{Pi}.
The t-topology exists for every group $\mathcal{G}$ definable
in an o-minimal structure $\mathcal{N}$ (not necessarily saturated),
and when $\mathcal{N}$ o-minimally expands the reals, $\mathcal{G}$
is a real Lie group \cite{Pi}.

A a locally definable subset $X$ of a locally definable group $\mathcal{G}$
is $\tau$-\textit{connected} if $X$ has no nonempty proper definable
subset ($\tau$-)clopen relative to $X$ such that whose intersection
with any definable subset of $\mathcal{G}$ is definable. When $\mathcal{G}$
is definable, by \cite[Corollary 2.10]{Pi}, there is a unique maximal
definably connected definable subset of $\mathcal{G}$ containing
the identity element of $\mathcal{G}$, which we call \textit{the definable identity component of}
$\mathcal{G}$, and we denoted it by $\mathcal{G}^{0}$. Thus, $\mathcal{G}$
is definably connected if and only if $\mathcal{G}=\mathcal{G}^{0}$,
or, equivalently by \cite{Pi}, if $\mathcal{G}$ has no proper definable
subgroup of finite index. We say that a definable group $\mathcal{G}$
is \textit{definably compact} if every definable path $\gamma:\left(0,1\right)\rightarrow\mathcal{G}$
has limits points in $\mathcal{G}$ (where the limits are taken with
respect to the t-topology on $\mathcal{G}$).

The notions of \textit{path connectedness, homotopy, o-minimal fundamental
group, and simply connectedness} are defined as in algebraic topology
using locally definable maps instead of general maps. For details
on these definitions we refer the reader to \cite{BarOter}.

As in the Lie setting \cite{Chevalley}, Edmundo and Eleftheriou defined and proved
in \cite{EdPan} the existence of the \textit{o-minimal universal covering
group} $\widetilde{\mathcal{G}}$ of a connected locally definable group
$\mathcal{G}$. As in Lie groups, $\widetilde{\mathcal{G}}$ covers
any cover of $\mathcal{G}$ in this category, is simply connected, and allows to study
definable and locally definable groups through them.

In the theory of Lie groups, it is known that when two Lie groups
are locally isomorphic, then their universal covers are (globally)
isomorphic as Lie groups:

\begin{fact}(\cite[Corollary 4.20]{MimuraToda})\label{F:LieGpsLocIsoUCovGlobIso}
Let $\mathcal{H}$ and $\mathcal{H}^{\prime}$ be connected Lie groups,
and $\widetilde{\mathcal{H}}^{Lie}$ and $\widetilde{\mathcal{H}^{\prime}}^{Lie}$
their universal covering Lie groups respectively. Then $\mathcal{H}$ and $\mathcal{H}^{\prime}$
are locally isomorphic if and only if $\widetilde{\mathcal{H}}^{Lie}$ and
$\widetilde{\mathcal{H}^{\prime}}^{Lie}$ are isomorphic as Lie groups.
\end{fact}

We say that two topological groups $\mathcal{H}$ and $\mathcal{H}^{\prime}$
are \textit{locally homomorphic} if there are neighborhoods
$U$ and $U^{\prime}$ of the identities of $\mathcal{H}$ and $\mathcal{H}^{\prime}$ respectively
and a map $f:U\subseteq\mathcal{H}\rightarrow U^{\prime}\subseteq\mathcal{H}^{\prime}$
such that $f\left(hh^{\prime}\right)=f\left(h\right)f\left(h^{\prime}\right)$
whenever $h,h^{\prime}$, and $hh^{\prime}$ belong to $U$ \cite[Definition 2, Chap. 2, Section 7]{Chevalley};
such a map $f$ is called a \textit{local homomorphism} of $\mathcal{H}$
into $\mathcal{H}^{\prime}$, and if in addition $f$ is a homeomorphism,
$f$ is called a \textit{local isomorphism}, and $\mathcal{H}$ and
$\mathcal{H}^{\prime}$ are \textit{locally isomorphic}.

From a model-theoretical point of view, it is natural to ask whether  Fact \ref{F:LieGpsLocIsoUCovGlobIso} holds or not in the category of locally definable
groups.

For this, we will restrict the previous definitions to definable maps.
However, if \textit{$\mathcal{M}$} expands an ordered field $\left(R,<,+,0,\cdot,1\right)$,
the additive group $\left(R,+\right)$ is definably locally isomorphic
to the group $G=\left(\left[0,1\right),+_{mod\,1}\right)$ with addition
modulo $1$ through the map $f:(-\frac{1}{2},\frac{1}{2})\subseteq\left(R,+\right)\rightarrow G$,
$f\left(t\right)=t\,mod\,1$, but $\left(R,+\right)$ and
$\widetilde{G}=\left(\bigcup_{n\in\mathbb{N}}\left(-n,n\right),+\right)\leq\left(R,+\right)$
are not isomorphic as locally definable groups ($\left(R,+\right)$ is definable,
but $\widetilde{G}$ is not).

Fact \ref{F:LieGpsLocIsoUCovGlobIso} follows from a well-known result
for topological groups that assures that a local homomorphism between
topological groups with domain a connected neighborhood of the identity
element of a simply connected group $\mathcal{H}$ can be extended
to a group homomorphism from the whole group $\mathcal{H}$:

\begin{fact}(\cite[Thm. 3, Chap. 2, Section 7]{Chevalley})\label{F:ExtLocHomoSCGp}
Let $\mathcal{H}$ be a simply connected topological group. Let $f$
be a local homomorphism of $\mathcal{H}$ into a group $\mathcal{H}^{\prime}$.
If the set on which $f$ is defined is connected, then it is possible
to extend $f$ to a homomorphism $\overline{f}:\mathcal{H}\rightarrow\mathcal{H}^{\prime}$.

\end{fact}

Again the above example of $\left(R,+\right)$ and $G=\left(\left[0,1\right),+_{mod\,1}\right)$
with the definable local isomorphism $f:(-\frac{1}{2},\frac{1}{2})\subseteq\left(R,+\right)\rightarrow G$,
$f\left(t\right)=t\,mod\,1$ shows that $f$ cannot be extended
to a locally definable homomorphism from $\left(R,+\right)$ into $G$
(otherwise, the kernel, $\mathbb{Z}$, of such a (locally) definable
homomorphism would be definable in $\mathcal{M}$, which is not possible).
Therefore, Fact \ref{F:ExtLocHomoSCGp} does not hold in the category
of locally definable groups.

Nevertheless, we were able to formulate a sufficient condition for a definable local homomorphism to extend to a locally definable homomorphism of the whole group. More precisely, we prove the following.

\begin{customthm}{\ref{T:ExtofdefLocHomoSCGp}}
\textit{Let $\mathcal{G}$ and $\mathcal{G}^{\prime}$ be locally
definable groups, $U$ a definably connected definable neighborhood
of the identity $e$ of $\mathcal{G}$, and $f:U\subseteq\mathcal{G}\rightarrow\mathcal{G}^{\prime}$
a definable local homomorphism. Assume that there is a definable neighborhood $V$ of $e$ generic in $\mathcal{G}$ such that $V^{-1}V\subseteq U$.
If $\mathcal{G}$ is simply connected, then $f$ is uniquely extendable
to a locally definable group homomorphism $\overline{f}:\mathcal{G}\rightarrow\mathcal{G}^{\prime}$.}
\end{customthm}

Above a subset $X$ of a group $\mathcal{G}$, locally definable
in a $\kappa$-saturated o-minimal structure, is
\textit{left (right) generic} in $\mathcal{G}$ if less than $\kappa$-many
left (right) group translates of $X$ cover $\mathcal{G}$; i.e.,
$\mathcal{G}=AX$ ($\mathcal{G}=XA$) for some $A\subseteq\mathcal{G}$
with $\left|A\right|<\kappa$. $X$ is \textit{generic} if it is
both left and right generic. A locally definable group may not have
definable generic subsets; however, when it does, the group has interesting properties, see for example \cite[Thm. 3.9]{PPant12}.

Theorem \ref{T:ExtofdefLocHomoSCGp} allows us to easily prove Corollary
\ref{C:ExtlocHomLDGpsTFcase}, a result on extension of a definable
local homomorphism between abelian locally definable groups that we
have previously proved in \cite[Thm. 9.1]{BADefCompIJM} using different methods.

\textit{From now on until the end of this paper, let $R=\left(R,<,+,\cdot\right)$
be a sufficiently saturated real closed field, and denote by $R_{a}$
its additive group $\left(R.+\right)$ and by $R_{m}$ its multiplicative
group of positive elements $\left(R^{>0},\cdot\right)$.}

Corollary \ref{C:ExtlocHomLDGpsTFcase} was applied in \cite{BADefCompIJM}
to prove a characterization of the o-minimal universal covering group $\widetilde{G}$ of an abelian definably connected definably compact
semialgebraic group $G$ over $R$ in terms of $R$-algebraic groups \cite[Thm. 10.2]{BADefCompIJM}
(see Fact \ref{F:AbelianlocHomLDGpsHaveisosubgintheirUniveCovers}).

In Section \ref{SS:UniveCoverCompactGeneral}, we show that Theorem
10.2 in \cite{BADefCompIJM} also holds for any definably
connected definably compact semialgebraic group over $R$ not necessarily
abelian (Thm. \ref{T:UniveCoverCompactGeneral}).

\begin{customthm}{\ref{T:UniveCoverCompactGeneral}}
\textit{Let $G$ be a definably connected definably compact group
definable in $R$. Then $\widetilde{G}$ is an open locally definable
subgroup of $\widetilde{H\left(R\right)^{0}}$ for some $R$-algebraic
group $H$.}
\end{customthm}

By \cite{PS05}, every abelian torsion free semialgebraic group over
$R$ is definably isomorphic to $R_{a}^{k}\times R_{m}^{n}$ for some
$k,n\in\mathbb{N}$. Therefore, the results for torsion free and definably compact semialgebraic groups over $R$ suggest asking ourselves the following.

\begin{question}\label{Q:OnAbelianSemiAlgUnivCovSubgpsAlgGps}
Let $G$ be an abelian definably connected semialgebraic group over
$R$. Is $\widetilde{G}$ an open locally definable subgroup of $\widetilde{H\left(R\right)^{0}}$ for some $R$-algebraic group $H$?
\end{question}

Although the above question remains open, we were able to prove:

\begin{customthm}{\ref{T:UniveCoverAbelianGeneral}}
\textit{Let $G$ be a definably connected semialgebraic group over $R$. Then there exist a locally definable group $\mathcal{W}$, commutative $R$-algebraic groups $H_{1}$, $H_{2}$ such that $\widetilde{G}$ is a locally definable extension of $\mathcal{W}$ by $H_{2}\left(R\right)^{0}$ where $\mathcal{W}$ is an open subgroup of $\widetilde{H_{1}\left(R\right)^{0}}$. In fact, $H_{2}\left(R\right)^{0}$ is isomorphic to $R_{a}^{k}\times R_{m}^{n}$ as definable groups for some
$k,n\in\mathbb{N}$.}
\end{customthm}

Where we say that a locally definable group $\mathcal{G}$ is a\textit{
locally definable extension} of $\mathcal{G}^{\prime}$ by $\mathcal{G}^{\prime\prime}$
if we have an exact sequence $1\rightarrow\mathcal{G}^{\prime\prime}\rightarrow\mathcal{G}\rightarrow\mathcal{G}^{\prime}\rightarrow1$
in the category of locally definable groups with locally definable
homomorphisms \cite[Section 4]{Ed06}.

\section{An extension of a definable local homomorphism between locally definable groups}\label{S:ExtLocHomomoSCLDGp}

Recall that we are working in a sufficiently $\kappa$-saturated o-minimal
structure $\mathcal{M}$.

\begin{thm}\label{T:ExtofdefLocHomoSCGp}
Let $\mathcal{G}$ and $\mathcal{G}^{\prime}$ be locally definable
groups, $U$ a definably connected definable neighborhood of the identity
$e$ of $\mathcal{G}$, and $f:U\subseteq\mathcal{G}\rightarrow\mathcal{G}^{\prime}$
a definable local homomorphism. Assume that there is a definable neighborhood $V$ of $e$, generic in $\mathcal{G}$, such that $V^{-1}V\subseteq U$. If $\mathcal{G}$ is simply connected, then $f$ is uniquely extendable
to a locally definable group homomorphism $\overline{f}:\mathcal{G}\rightarrow\mathcal{G}^{\prime}$.
\end{thm}

\begin{proof}

Let $x\in\mathcal{G}$. Since $\mathcal{G}$ is path connected, there
is a locally definable path $\omega:I=[0,1]\rightarrow\mathcal{G}$
such that $\omega\left(0\right)=e$, $\omega\left(1\right)=x$. Note that since $V$ is generic in $\mathcal{G}$, then the topological interior of $V$ is also generic in $\mathcal{G}$ -- this fact might be followed from Lemma 2.22 of \cite{MarikovaJana}, for example -- . So replacing $V$ with its topological interior, we now have that $V$ is an open neighborhood of $e$ and generic in $\mathcal{G}$. By the genericity of $V$ in $\mathcal{G}$, $\mathcal{G}=A\cdot V$ for
some $A\subseteq\mathcal{G}$, $\left|A\right|<\kappa$. As $\omega\left(I\right)$
is a definable subset of $A\cdot V$, saturation yields that there
is a finite set $A_{0}\subseteq A$ such that $\omega\left(I\right)\subseteq A_{0}\cdot V$.
Then, $I=\bigcup_{a\in A_{0}}\omega^{-1}\left(a\cdot V\right)$. As
$V$ is an open neighborhood, $\omega^{-1}\left(a\cdot V\right)$
is also an open neighborhood of some element of $I$.

Since $\mathcal{M}$ is an o-minimal structure, $\omega^{-1}\left(a\cdot V\right)$
is a finite union of points and intervals. Then $\left\{ \omega^{-1}\left(a\cdot V\right):a\in A_{0}\right\} $
is a collection of open intervals in $I$. Thus we can choose a division
of $I$ into subintervals $\left[t_{i},t_{i+1}\right]$ such that
$0=t_{0}<t_{1}<\ldots<t_{n}=1$ and $\left(\omega\left[t_{i},t_{i+1}\right]\right)\subseteq a_{i}V$
for some $a_{i}\in A_{0}$. So, for $t,t^{\prime}\in\left[t_{i},t_{i+1}\right]$,
$\omega\left(t\right)=a_{i}v$, $\omega\left(t^{\prime}\right)=a_{i}v^{\prime}$
for $v,v^{\prime}\in V$, and $\omega\left(t\right)^{-1}\omega\left(t^{\prime}\right)=v^{-1}v^{\prime}\in V^{-1}V\subseteq U$.

For such a locally definable path $\omega$ and division define

\[
\overline{f}_{\omega}\left(x\right)\coloneqq f\left(\omega\left(t_{0}\right)^{-1}\omega\left(t_{1}\right)\right)f\left(\omega\left(t_{1}\right)^{-1}\omega\left(t_{2}\right)\right)\cdots f\left(\omega\left(t_{n-1}\right)^{-1}\omega\left(t_{n}\right)\right).
\]

Now, we will show that $\overline{f}_{\omega}$ is invariant under
refinements of the division of $I$.

Let $t^{\prime}\in\left[t_{i},t_{i+1}\right]$ be a new subdivision point. Since $\omega\left(t_{i}\right)^{-1}\omega\left(t^{\prime}\right),\omega\left(t^{\prime}\right)^{-1}\omega\left(t_{i+1}\right)\in U$
and $f$ is a local homomorphism, $f\left(\omega\left(t_{i}\right)^{-1}\omega\left(t^{\prime}\right)\right)f\left(\omega\left(t^{\prime}\right)^{-1}\omega\left(t_{i+1}\right)\right)=f\left(\omega\left(t_{i}\right)^{-1}\omega\left(t_{i+1}\right)\right)$.

Hence, given two subdivisions of $I$, we can consider a refinement
common to these, then $\overline{f}_{\omega}$ does not depend on
the subdivisions of $I$.

Now, we will show that $\overline{f}_{\omega}$ is determined independently
of the choice of a path $\omega$. Let $\omega^{\prime}:I\rightarrow\mathcal{G}$
be another locally definable path connecting $e$ and $x$. Since
$\mathcal{G}$ is simply connected, there is a locally definable homotopy
$\Gamma:I\times I\rightarrow\mathcal{G}$ between $\omega$ and $\omega^{\prime}$
with $\Gamma\left(t,0\right)=\omega\left(t\right)$ and $\Gamma\left(t,1\right)=\omega^{\prime}\left(t\right)$.
As $\Gamma\left(I\times I\right)$ is a definable subset of $A\cdot V$,
again saturation implies that there is a finite set $A_{1}\subseteq A$
such that $\Gamma\left(I\times I\right)\subseteq A_{1}\cdot V$. So,
$I\times I\subseteq\bigcup_{a\in A_{1}}\Gamma^{-1}\left(a\cdot V\right)$.
By continuity of $\varGamma$, for every $\left(t,t^{\prime}\right)\in I\times I$,
there is $I_{i}\times I_{j}\subseteq I\times I$ such that $\left(t,t^{\prime}\right)\in I_{i}\times I_{j}$,
$\Gamma\left(I_{i}\times I_{j}\right)\subseteq a_{i,j}V$ for some
$a_{i,j}\in A_{1}$. Therefore, we can partition $I$ in a finite
number of subintervals $I_{i}$ such that $\Gamma\left(I_{i}\times I_{j}\right)\subseteq a_{i,j}V$
for some $a_{i,j}\in A_{1}$. As the subintervals $I_{i}$ are finite
and cover $I$, we may assume that $I_{i}=\left[\frac{i}{n},\frac{i+1}{n}\right]$
for some $n\in\mathbb{N}$. Note that if $\left(s,t\right),\left(s^{\prime},t^{\prime}\right)\in I_{i}\times I_{j}$,
then $\varGamma\left(s,t\right)^{-1}\varGamma\left(s^{\prime},t^{\prime}\right)=\left(a_{i,j}v\right)^{-1}a_{i,j}v^{\prime}=v^{-1}v^{\prime}\in U$.
Let $\omega_{i}\left(t\right)\coloneqq\varGamma\left(t,\frac{i}{n}\right)$
for $i\in\left\{ 0,1,\ldots,n\right\} $. So $\omega_{0}\left(t\right)=\omega\left(t\right)$,
$\omega_{n}\left(t\right)=\omega^{\prime}\left(t\right)$. Since $f$
is a local homomorphism, then $\overline{f}_{\omega_{i}}\left(x\right)=\overline{f}_{\omega_{i+1}}\left(x\right)$
for $i\in\left\{ 0,1,\ldots,n-1\right\} $. Therefore, $\overline{f}_{\omega}$
is determined independently of the choice of a path $\omega$, and
denote it by $\overline{f}$ .

Now, let $x,y\in\mathcal{G}$ and $\omega,\gamma:I\rightarrow\mathcal{G}$
locally definable paths connecting $e$ and $x$, and $e$ and $y$,
respectively. Then the locally definable path $\sigma:I\rightarrow\mathcal{G}$, $\sigma\left(t\right)\coloneqq x\gamma\left(t\right)$ connects
$x$ with $xy$. Let $\omega*\sigma$ denote the concatenation of
the paths $\omega$ and $\sigma$. Then, $\overline{f}_{\omega}\left(x\right)\overline{f}_{\gamma}\left(y\right)=\overline{f}_{\omega*\sigma}\left(xy\right)$,
namely $\overline{f}\left(x\right)\overline{f}\left(y\right)=\overline{f}\left(xy\right)$,
so $\overline{f}$ is a group homomorphism.

Next, we will see that $\overline{f}$ is an extension of $f$. As
$U$ is definably connected, so is path connected \cite{BarOter},
then if $x\in U$, there is a locally definable path $\omega:I\rightarrow\mathcal{G}$
such that $\omega\left(0\right)=e$, $\omega\left(1\right)=x$, and
$\omega\left(I\right)\subseteq U$. As $\overline{f}$ does not depend
on the subdivisions of $I$, let $t_{0}=0$, $t_{1}=1$, then it is
clear that $\overline{f}\left(x\right)=f\left(\omega\left(t_{0}\right)^{-1}\omega\left(t_{1}\right)\right)=f\left(x\right)$.

Now, let $h$ be another extension of $f$. Let $\overline{\mathcal{G}}=\left\{ g\in\mathcal{G}:h\left(g\right)=\overline{f}\left(g\right)\right\} $.
Then $\overline{\mathcal{G}}$ is a locally definable subgroup of
$\mathcal{G}$, and $U\subseteq\overline{\mathcal{G}}$. As $U$ is
generic in $\mathcal{G}$, $U$ generates $\mathcal{G}$ \cite[Fact 2.3(2)]{PPant12},
then $\mathcal{G}\subseteq\overline{\mathcal{G}}$, so $h=\overline{f}$.
Therefore, $f$ is uniquely extendable to a locally definable group
homomorphism from $\mathcal{G}$ into $\mathcal{G}^{\prime}$.

Finally, observe that $\overline{f}$ is a locally definable map on
$\mathcal{G}$. For this note that since $\overline{f}$ is a group
homomorphism, $\overline{f}\left(x^{-1}y\right)=\overline{f}\left(x\right)^{-1}\overline{f}\left(y\right)=f\left(x\right)^{-1}f\left(y\right)$
for every $x,y\in U$, then $\overline{f}$ restricted to $\prod_{n}U^{-1}U$
is a definable map. And as $\mathcal{G}=\left\langle U\right\rangle =\bigcup_{n\in\mathbb{N}}\prod_{n}U^{-1}U$,
then $\overline{f}$ is a locally definable map on $\mathcal{G}$.
\end{proof}

For a locally definable group $\mathcal{G}$ in $\mathcal{M}$, we
denote by $\mathcal{G}^{00}$ the smallest, if such exists, type-definable
subgroup of $\mathcal{G}$ of index smaller than $\kappa$, where
for a small set we mean a subset of $M^{n}$ with cardinality smaller
than $\kappa$ (\cite{HPePiNIP}). $\mathcal{G}^{00}$ may not exist
(see an example in \cite[Subsection 2.2]{PPant12}). For definable groups such a type-definable group always exists \cite{Shelah}.

With Theorem \ref{T:ExtofdefLocHomoSCGp}, it is easy to prove the
following result, which was previously demonstrated in \cite{BADefCompIJM}
using a different technique.

\begin{cor}\label{C:ExtlocHomLDGpsTFcase} (\cite[Thm. 9.1]{BADefCompIJM})
Let $\mathcal{G}$ and $\mathcal{G}^{\prime}$ be two abelian locally
definable groups such that $\mathcal{G}$ is connected, torsion free,
and $\mathcal{G}^{00}$ exists. Let $U\subseteq\mathcal{G}$ be a
definably connected definable set such that $\mathcal{G}^{00}\subseteq U$,
and $f:U\subseteq\mathcal{G}\rightarrow\mathcal{G}^{\prime}$ a definable
local homomorphism. Then $f$ is uniquely extendable to a locally
definable group homomorphism $\overline{f}:\mathcal{G}\rightarrow\mathcal{G}^{\prime}$.
\end{cor}

\begin{proof}

We will just check that the assumptions are enough to apply Theorem
\ref{T:ExtofdefLocHomoSCGp}. First, we will see that $\mathcal{G}$
is simply connected. Since $\mathcal{G}^{00}$ exists, then $\mathcal{G}$
is definably generated, and by \cite[Thm 3.9]{PPant12}, $\mathcal{G}$
covers an abelian definable group. Claim 6.4 in \cite{BADefCompIJM}
yields that $\mathcal{G}$ is simply connected since $\mathcal{G}$
is also torsion free. Now, as $\mathcal{G}^{00}\subseteq U$, by saturation,
there is a definable set $\overline{V}$ such that $\mathcal{G}^{00}\subseteq\overline{V}\subseteq\overline{V}^{-1}\overline{V}\subseteq U$.
Let $V$ be the topological interior of $\overline{V}$ in $\mathcal{G}$,
which is definable, then $\mathcal{G}^{00}\subseteq V\subseteq V^{-1}V\subseteq U$
since $\mathcal{G}^{00}$ is open in $\mathcal{G}$. Therefore, $V$
is a definable neighborhood of the identity in $\mathcal{G}$ open
in $\mathcal{G}$, generic in $\mathcal{G}$, and $V^{-1}V\subseteq U$.
Finally, Theorem \ref{T:ExtofdefLocHomoSCGp} implies that there is
a unique extension $\overline{f}:\mathcal{G}\rightarrow\mathcal{G}^{\prime}$
$f$ that is a locally definable group homomorphism.
\end{proof}

\section{Universal covers of definably local homomorphic locally definable groups}\label{S:UnivCovOfLocHomomoLDGps}

As in \cite{BADefCompIJM}, Corollary \ref{C:ExtlocHomLDGpsTFcase},
together with other results in \cite{BADefCompIJM}, implies the following
fact, and establishes a relation between the o-minimal universal
covering groups of two definably locally homomorphic locally definable groups. This result could be interpreted as an analogue
in the category of locally definable groups of the known fact that two connected locally homomorphic Lie groups have isomorphic universal covering Lie groups (Fact \ref{F:LieGpsLocIsoUCovGlobIso}).

\begin{fact}(\cite[Thm. 10.1]{BADefCompIJM})\label{F:AbelianlocHomLDGpsHaveisosubgintheirUniveCovers}
Let $\mathcal{G}$ and $\mathcal{G}^{\prime}$ be two divisible abelian
connected locally definable groups such that $\mathcal{G}^{00}$ exists
and $\mathcal{G}^{00}$ is a decreasing intersection of $\omega$-many
simply connected definable subsets of $\mathcal{G}$.

Let $X\subseteq\mathcal{G}$ be a definable set with $\mathcal{G}^{00}\subseteq X$,
and $f:X\subseteq\mathcal{G}\rightarrow\mathcal{G}^{\prime}$ a definable
homeomorphism and local homomorphism. Then $\widetilde{\mathcal{G}}$
is an open locally definable subgroup of $\widetilde{\mathcal{G}^{\prime}}$.

\end{fact}

An important corollary of the above result is the characterization
of the o-minimal universal covers of the abelian definably
connected definably compact semialgebraic groups over $R$
in terms of algebraic groups.

\begin{fact}(\cite[Thm. 10.2]{BADefCompIJM})\label{F:UniveCoverAbelianCompact}
Let $G$ be an abelian definably connected definably compact group
definable in $R$. Then $\widetilde{G}$ is an open locally definable
subgroup of $\widetilde{H\left(R\right)^{0}}$ for some Zariski-connected
$R$-algebraic group $H$.
\end{fact}

In fact, the algebraic group $H$ in Fact \ref{F:UniveCoverAbelianCompact}
is commutative since $G$ is abelian, and by Theorem 4.1 in \cite{BADefCompIJM}, $G$ and $H\left(R\right)^{0}$ are definably locally homomorphic.

\subsection{Universal cover of a definably compact semialgebraic group}\label{SS:UniveCoverCompactGeneral}

Now, we will show that Fact \ref{F:UniveCoverAbelianCompact} also
holds for every definably connected definably compact $R$-definable group  not necessarily
abelian. For this, we will use several results of Hrushovski, Peterzil,
and Pillay in \cite{HPP2011} together with the abelian case Fact
\ref{F:UniveCoverAbelianCompact}.

\begin{thm}\label{T:UniveCoverCompactGeneral}
Let $G$ be a definably connected definably compact group definable
in $R$. Then $\widetilde{G}$ is an open locally definable subgroup
of $\widetilde{H\left(R\right)^{0}}$ for some $R$-algebraic group
$H$.
\end{thm}
\begin{proof}
By \cite[Corollary 6.4]{HPP2011}, $G$ is definably isomorphic to
the almost direct product of $S$ and $G_{0}$ where $S$ is some
definably connected semisimple definable group, and $G_{0}$ is some
abelian definably connected definably compact group, so $G\simeq\left(G_{0}\times S\right)/F$
for some finite central subgroup $F\subseteq G_{0}\times S$. Therefore,
$G_{0}\times S$ is a finite cover of $G$, then $\widetilde{G}\simeq\widetilde{G_{0}\times S}$.
Since the o-minimal fundamental group for locally definable groups
(see \cite{BarOter} ) has the property that the group $\pi_{1}\left(\mathcal{G}_{1},g_{1}\right)\times\pi_{1}\left(\mathcal{G}_{2},g_{2}\right)$
is isomorphic to the group $\pi_{1}\left(\mathcal{G}_{1}\times\mathcal{G}_{2},\left(g_{1},g_{2}\right)\right)$
for $\mathcal{G}_{1}$, $\mathcal{G}_{2}$ locally definable groups
and $g_{1}$, $g_{2}$ elements in $\mathcal{G}_{1}$, $\mathcal{G}_{2}$,
respectively, then if $\mathcal{G}_{1}$ and $\mathcal{G}_{2}$ are
simply connected, then $\widetilde{\mathcal{G}_{1}\times\mathcal{G}_{2}}$
is isomorphic to $\widetilde{\mathcal{G}_{1}}\times\widetilde{\mathcal{G}_{2}}$
as locally definable groups. Hence, $\widetilde{G_{0}\times S}\simeq\widetilde{G_{0}}\times\widetilde{S}$.

Now, by \cite[Fact 1.2(3)]{HPP2011}, the center $Z\left(S\right)$
of $S$ is finite and $S/Z\left(S\right)$ is definably isomorphic
to a direct product $S_{1}\times\ldots\times S_{n}$ of finitely many
definably simple groups $S_{i}$'s. By \cite[Fact 1.2(1)]{HPP2011},
$S_{i}\simeq H_{i}\left(R_{i}\right)^{0}$ for some real
closed field $R_{i}$ definable in $R$ and some $R_{i}$-algebraic group $H$.
But, by \cite[Thm. 1.1]{Otero1996OnGA}, $R_{i}\simeq R$. Let $H_{\star}\coloneqq H_{1}\times\ldots\times H_{n}$,
then $H_{\star}\left(R\right)^{0}\simeq H_{1}\left(R\right)^{0}\times\ldots\times H_{n}\left(R\right)^{0}\simeq S_{1}\times\ldots\times S_{n}$;
namely, every semisimple semialgebraic group $S$ over $R$ is up
to its center $H_{\star}\left(R\right)^{0}$ for some $R$-algebraic
group $H_{\star}$. Thus, $\widetilde{S}\simeq\widetilde{H_{\star}\left(R\right)^{0}}$.

Now, by Fact \ref{F:UniveCoverAbelianCompact}, $\widetilde{G_{0}}\leq\widetilde{H_{\star\star}\left(R\right)^{0}}$ for
some $R$-algebraic group $H_{\star\star}$, then $\widetilde{G}\simeq\widetilde{G_{0}}\times\widetilde{S}\leq\widetilde{H_{\star\star}\left(R\right)^{0}}\times\widetilde{H_{\star}\left(R\right)^{0}}\simeq\widetilde{H\left(R\right)^{0}}$
for $H\coloneqq H_{\star\star}\times H_{\star}$. Note that $\widetilde{G}$
and $\widetilde{H\left(R\right)^{0}}$ have the same dimension, then
$\widetilde{G}$ is, up to isomorphism of locally definable groups,
an open locally definable subgroup of $\widetilde{H\left(R\right)^{0}}$ for
some $R$-algebraic group $H$.

\end{proof}

\subsection{Universal cover of an abelian semialgebraic group}\label{SS:UniveCoverAbelianGeneral}

In this subsection, we will prove that the o-minimal universal covering
group of an abelian definably connected semialgebraic group over $R$
is a locally definable (group) extension, in the category of locally
definable groups (see \cite[Section 4]{Ed06} for basics on locally
definable extensions), of an open locally definable subgroup of $\widetilde{H_{1}\left(R\right)^{0}}$
by $H_{2}\left(R\right)^{0}$ for some $R$-algebraic groups $H_{1}$,
$H_{2}$. This will mainly follow by the characterization of abelian
groups definable in o-minimal structures \cite{Edmundo} and our previous
results in this work.

Recall that for the sufficiently saturated real closed field $R=\left(R,<,+,\cdot\right)$,
$R_{a}$ denotes the additive group $\left(R.+\right)$, and $R_{m}$
the multiplicative group of positive elements $\left(R^{>0},\cdot\right)$.

\begin{thm}\label{T:UniveCoverAbelianGeneral}
Let $G$ be a definably connected semialgebraic group over $R$. Then there exist a locally definable group $\mathcal{W}$, commutative $R$-algebraic groups $H_{1}$, $H_{2}$ such that $\widetilde{G}$ is a locally definable extension of $\mathcal{W}$ by $H_{2}\left(R\right)^{0}$ where $\mathcal{W}$ is an open subgroup of $\widetilde{H_{1}\left(R\right)^{0}}$. In fact, $H_{2}\left(R\right)^{0}$ is isomorphic to $R_{a}^{k}\times R_{m}^{n}$ as definable groups for some
$k,n\in\mathbb{N}$.
\end{thm}

\begin{proof}

By \cite{Edmundo}, $G$ is a definable extension of some abelian
definably compact definably connected definable group $K$ by the
maximal torsion free normal definable subgroup $T$ of $G$: $1\rightarrow T\rightarrow G\rightarrow K\rightarrow1$. By \cite[Thm. 10.2]{BADefCompIJM},
$\widetilde{K}\leq\widetilde{H_{1}\left(R\right)^{0}}$ for some commutative
$R$-algebraic group $H_{1}$.

By \cite{PS05}, $T$ is definably isomorphic to $R_{a}^{k}\times R_{m}^{n}$
for some $k,n\in\mathbb{N}$. So in particular, $T\simeq H_{2}\left(R\right)^{0}$ for $H_{2}=\left(R\left(\sqrt{-1}\right),+\right)^{k}\times\left(R\left(\sqrt{-1}\right),\cdot\right)^{n}$.
Thus, so far we have that $1\rightarrow H_{2}\left(R\right)^{0}\rightarrow G\overset{\pi}{\rightarrow}K\rightarrow1$
with $\widetilde{K}\leq\widetilde{H_{1}\left(R\right)^{0}}$ for some
commutative $R$-algebraic groups $H_{1},H_{2}$.

By \cite[Thm. 5.1]{PS05}, there is a continuous definable section
$s:K\rightarrow G$ (continuous with respect with their $\tau$-topologies). Then the map $\varphi:H_{1}\left(R\right)^{0}\times K\rightarrow G$,
$\varphi\left(h,k\right)=hs\left(k\right)$ is a definable homeomorphism
with inverse $\varphi^{-1}\left(g\right)=\left(g\left(s\left(\pi\left(g\right)\right)\right)^{-1},\pi\left(g\right)\right)$
for $g\in G$. Here the direct product $H_{1}\left(R\right)^{0}\times K$
has the product topology, and the groups $K$, $G$, and the subgroup
$H_{1}\left(R\right)^{0}\leq G$ have the $\tau$-topology (\cite{PS00})
which coincides with the t-topology (\cite{Pi}) for definable groups.

Let $f$ be the definable two-cocycle associated with the section $s$,
i.e.,
\[
\begin{aligned}f:K\times K & \rightarrow H_{1}\left(R\right)^{0}\\
\left(k_{1},k_{2}\right) & \mapsto s\left(k_{1}\right)s\left(k_{2}\right)\left(s\left(k_{1}k_{2}\right)\right)^{-1}_{.}
\end{aligned}
\]

Then, $G$ is definably isomorphic to the group $\left(H_{1}\left(R\right)^{0}\times K,\cdot_{f}\right)$
with group operation given by
\[
\left(h,k\right)\cdot_{f}\left(h^{\prime},k^{\prime}\right)=\left(hh^{\prime}f\left(k,k^{\prime}\right),kk^{\prime}\right),
\] through the definable group isomorphism $\varphi$.

Let $p_{K}:\widetilde{K}\rightarrow K$ be the o-minimal universal
covering homomorphism of $K$, and $\textrm{id}:H_{1}\left(R\right)^{0}\rightarrow H_{1}\left(R\right)^{0}$
the identity map on $H_{1}\left(R\right)^{0}$.

Now, let
\[
\begin{aligned}\widetilde{f}:\widetilde{K}\times\widetilde{K} & \rightarrow H_{1}\left(R\right)^{0}\\
\left(\widetilde{k_{1}},\widetilde{k_{2}}\right) & \mapsto f\left(p_{K}\left(\widetilde{k_{1}}\right),p_{K}\left(\widetilde{k_{2}}\right)\right)
\end{aligned}
.\]

The two-cocycle condition (\cite[Eq. 3, Section 3]{Edmundo}) of $f$ implies the same condition for $\widetilde{f}$, thus the group $\left(H_{1}\left(R\right)^{0}\times\widetilde{K},\cdot_{\widetilde{f}}\right)$
with group operation given by
\[
\left(h,\widetilde{k_{1}}\right)\cdot_{\widetilde{f}}\left(h^{\prime},\widetilde{k_{2}}\right)=\left(hh^{\prime}\widetilde{f}\left(\widetilde{k_{1}},\widetilde{k_{2}}\right),\widetilde{k_{1}}\widetilde{k_{2}}\right)
\] induced by $\widetilde{f}$ is a locally definable group. Let $i:H_{1}\left(R\right)^{0}\rightarrow H_{1}\left(R\right)^{0}\times\widetilde{K}$
be the map $h\mapsto\left(h,1\right)$,
and $\pi_{2}:H_{1}\left(R\right)^{0}\times\widetilde{K}\rightarrow\widetilde{K}$
the projection map into the second coordinate. So far we have that
\[
1\rightarrow H_{2}\left(R\right)^{0}\overset{i}{\rightarrow}\left(H_{1}\left(R\right)^{0}\times\widetilde{K},\cdot_{\widetilde{f}}\right)\overset{\pi_{2}}{\rightarrow}\widetilde{K}\rightarrow1
\] is a locally definable extension. Note that the locally definable
group $\left(H_{1}\left(R\right)^{0}\times\widetilde{K},\cdot_{\widetilde{f}}\right)$
is connected because $H_{1}\left(R\right)^{0}$ and $\widetilde{K}$
are both connected \cite[Corollary 4.8(ii)]{Ed06}.

Now, the map
\[
\begin{aligned}\varphi\circ\left(\textrm{id}\times p_{K}\right):\left(H_{1}\left(R\right)^{0}\times\widetilde{K},\cdot_{\widetilde{f}}\right) & \rightarrow G\\
\left(h,\widetilde{k}\right) & \mapsto\varphi\left(h,p_{K}\left(\widetilde{k}\right)\right)=hs\left(p_{K}\left(\widetilde{k}\right)\right)
\end{aligned}
\] is a locally definable covering map. The abelianity of $G$ and the
definition of $\widetilde{f}$ let easily to conclude that $\varphi\circ\left(\textrm{id}\times p_{K}\right)$
is also a group homomorphism. Hence, $\varphi\circ\left(\textrm{id}\times p_{K}\right):\left(H_{1}\left(R\right)^{0}\times\widetilde{K},\cdot_{\widetilde{f}}\right)\rightarrow G$
is a locally definable covering homomorphism.

Next, we will see that $\left(H_{1}\left(R\right)^{0}\times\widetilde{K},\cdot_{\widetilde{f}}\right)$
is simply connected. Note that, by \cite[Proposition 5.14]{BerarEdMamino},
$\widetilde{K}$ is torsion free. So $H_{1}\left(R\right)^{0}$ and
$\widetilde{K}$ are both torsion free, and therefore $\left(H_{1}\left(R\right)^{0}\times\widetilde{K},\cdot_{\widetilde{f}}\right)$
is too. Finally, \cite[Claim 6.4]{BADefCompIJM} yields the simply
connectedness of the group $\left(H_{1}\left(R\right)^{0}\times\widetilde{K},\cdot_{\widetilde{f}}\right)$,
then \[\varphi\circ\left(\textrm{id}\times p_{K}\right):\left(H_{1}\left(R\right)^{0}\times\widetilde{K},\cdot_{\widetilde{f}}\right)\rightarrow G\] is the o-minimal universal covering homomorphism of $G$. We conclude
the proof of the theorem.
\end{proof}

\section*{Acknowledgements}

I warmly thank Assaf Hasson and Kobi Peterzil for their support, generous ideas, and kindness during this work. I also want to express my gratitude to the Ben-Gurion University of the Negev, Israel joint with its warm and helpful staff for supporting my research.

The main results of this paper have been presented at the Israel Mathematical Union (IMU) virtual meeting 2020 on September 6th, 2020 (Israel), the Logic virtual seminar of the Università degli Studi della Campania ``Luigi Vanvitelli'' on May 28th, 2020 (Campania, Italy), and the Logic seminar of the Hebrew University of Jerusalem (Jerusalem, Israel) on December 11th, 2019.

This research was funded by the Israel Science Foundation (ISF) grants No. 181/16 and 1382/15.

\nocite{BaEdErrata}
\bibliographystyle{plain}
\bibliography{IP}

\end{document}